\documentclass[a4paper, reqno]{amsart}
\usepackage[margin=2.5cm]{geometry}
\usepackage{amsmath,amssymb,amsthm, enumerate, mathtools}
\usepackage[dvipsnames]{xcolor}
\usepackage[all,cmtip]{xy}
\usepackage[alphabetic,nobysame, initials]{amsrefs}
\usepackage{tikz}
\usepackage{graphicx}
		\graphicspath{{Pictures/} }
\usepackage{color} 
\definecolor{darkblue}{rgb}{0,0,0.6}

\usepackage[breaklinks, pdftex, ocgcolorlinks,colorlinks=true, citecolor=darkblue, filecolor=darkblue, linkcolor=darkblue, urlcolor=darkblue]{hyperref}
\usepackage[capitalize,noabbrev]{cleveref}

\newcommand{\RR}{\mathbb{R}}

\newcommand{\Z}{\mathbb{Z}}

\newcommand{\CP}{\mathbb{CP}}

\newcommand{\ol}{\overline}
\newcommand{\wt}{\widetilde}

\newcommand{\sm}{\smallsetminus}
\renewcommand{\epsilon}{\varepsilon}
\DeclareMathOperator{\Arf}{Arf}

\newcommand{\Sum}{\displaystyle \sum}

\renewcommand{\top}{\mathrm{top}}

\newcommand{\varG}{\widetilde{G}_{\CP^2}}

\makeatletter
\newtheorem*{rep@theorem}{\rep@title}
\newcommand{\newreptheorem}[2]{%
\newenvironment{rep#1}[1]{%
 \def\rep@title{#2 \ref{##1}}%
 \begin{rep@theorem}}%
 {\end{rep@theorem}}}
\makeatother

\theoremstyle{plain}
\newtheorem{theorem}{Theorem}[section]
	\newtheorem{proposition}[theorem]{Proposition}
	\newtheorem{lemma}[theorem]{Lemma}
	\newtheorem{corollary}[theorem]{Corollary}

\newreptheorem{theorem}{Theorem}
\newreptheorem{lemma}{Lemma}
\newreptheorem{proposition}{Proposition}
\newreptheorem{corollary}{Corollary}

\theoremstyle{definition}

	\newtheorem{remark}[theorem]{Remark}

\numberwithin{figure}{section}
\numberwithin{equation}{section}

\begin{document}

\title{A note on surfaces in  $\CP^2$
and $\CP ^2\# \CP ^2$}

\author[M.~Marengon]{Marco Marengon}
\address{HUN-REN R\'enyi Institute of Mathematics,    Re\'altanoda utca 13-15, 1053. Budapest, Hungary}
\email{marengon@renyi.hu}

\author[A.\ N.\ Miller]{Allison N.\ Miller}
\address{Department of Mathematics \& Statistics, Swarthmore College, 500 College Avenue, Swarthmore, PA. 19081}
\email{amille11@swarthmore.edu }

\author[A.~Ray]{Arunima Ray}
\address{Max Planck Institut f\"{u}r Mathematik, Vivatsgasse 7, 53111 Bonn, Germany}
\email{aruray@mpim-bonn.mpg.de }

\author[A.\ I.\ Stipsicz]{Andr\'as I.\ Stipsicz}
\address{HUN-REN R\'enyi Institute of Mathematics,    Re\'altanoda utca 13-15, 1053. Budapest, Hungary}
\email{stipsicz.andras@renyi.hu}

\def\subjclassname{\textup{2020} Mathematics Subject Classification}
\expandafter\let\csname subjclassname@1991\endcsname=\subjclassname

\begin{abstract}
In this brief note, we investigate the $\CP^2$-genus of knots, i.e.\ the least genus of a smooth, compact, orientable surface in $\CP^2\sm \mathring{B^4}$ bounded by a knot in $S^3$. We show that this quantity is unbounded, unlike its topological counterpart. We also investigate the $\CP^2$-genus of torus knots.  We apply these results to improve the minimal genus bound for some homology classes
in $\CP^2\# \CP  ^2$.
\end{abstract}
\maketitle

\section{Introduction}
The genus function $G_X$ of a smooth $4$-manifold $X$, as a function
$H_2(X; \Z )\to \Z_{\geq 0}$, is defined for $\alpha \in H_2(X ; \Z )$ as
\begin{equation}\label{eq:mingenus}
G_X(\alpha )=\min \{ g(F )\mid i\colon F \to X, i_*([F ])=\alpha\},
\end{equation}
where the minimum is taken over all smooth embeddings $i$ of smooth, closed,
oriented surfaces $F$. A triumph of modern gauge theory consists of Kronheimer and Mrowka's
determination of $G_{\CP^2}$, i.e.\ the solution to the minimal genus
problem in $\CP^2$, also called the Thom
conjecture~\cite{thomconjecture}. They showed that if $h\in H_2 (\CP
^2; \Z)\cong \Z $ is a generator and $d\neq 0$ is an integer, then
\[
G_{\CP ^2}(d\cdot h)=\frac{(\vert d\vert -1)(\vert d \vert -2)}{2},
\]
and $G_{\CP ^2}(0)=0$.
In this paper we will consider two generalisations of this: in one direction we will study a relative version of the genus bound, for surfaces with boundary, embedded in a punctured $\CP^2$ (Subsection \ref{sec:relativeCP2}); in another direction, we will examine the classical minimal smooth genus problem in the connected sum of \emph{two} copies of $\CP^2$ (Subsection \ref{sec:absolute2CP2}).

\subsection{Relative minimal genus problem in \texorpdfstring{$\CP^2$}{CP2}}
\label{sec:relativeCP2}
The
\emph{$\CP^2$-genus} of a knot $K\subseteq S^3$, denoted by
$g_{\CP^2}(K)$, is the least genus of a smooth, compact, connected,
orientable, properly embedded surface $F\subseteq \CP^2\sm
\mathring{B^4}=:(\CP^2)^\times$ bounded by $K\subseteq \partial
(\CP^2)^\times\cong S^3$. Similarly, there is the \emph{topological
$\CP^2$-genus}, denoted by $g^\top_{\CP^2}(K)$, the least genus of a
surface $F$ as above which is only locally flat and embedded. A
knot $K$ is said to be slice in $\CP^2$ if $g_{\CP^2}(K)=0$, and
topologically slice in $\CP^2$ if $g_{\CP^2}^\top(K)=0$.

It was shown in~\cite{KPRT}*{Corollary~1.12\,(2)} that $g^\top_{\CP^2}(K)\leq 1$ for every knot $K$. By contrast, we show that $g_{\CP^2}$ is unbounded.

\begin{theorem}\label{thm:top-v-smooth}\label{thm:arb-large}
Let $n\geq 0$. There exists a topologically slice knot $K$ 
with $g_{\CP^2}(K)\geq n$. 
\end{theorem}

This result answers \cite{aitnouh}*{Question~2.1}, which asked for a
knot which is topologically slice in $\CP^2$ but not smoothly
slice. The largest value of $g_{\CP^2}$ previously in the literature
was $g_{\CP^2}(T_{2,17})=7$ from~\cite{aitnouh}*{Theorem~1.2}. Our
examples can be taken to be connected sums of the untwisted, negative
clasped Whitehead double of the left-handed trefoil knot. Since these
are topologically slice in $B^4$, they have trivial $g_{\CP^2}^\top$.

On the constructive side, we also give a method to find knots with
trivial $g^\top_{\CP^2}$, that are not necessarily topologically slice
in $B^4$. 

\begin{theorem}\label{thm:arf-zero}
Let $K\subseteq S^3$ be a knot. If $\Arf(K)=0$ then $g^\mathrm{top}_{\CP^2}(K)=0$. 
\end{theorem}

Note that the above result does not give a complete characterization
of knots which are topologically slice in $\CP^2$, since for example
the right-handed trefoil $T$ satisfies
$g_{\CP^2}(T)=g^\top_{\CP^2}(T)=0$, as an unknotting number one knot,
but has $\Arf(T)=1$. The proof of \cref{thm:arf-zero} uses a result of
Freedman and Quinn on when an immersed disk with an algebraically dual
sphere is homotopic to a locally flat
embedding~\cite{FQ}*{Theorem~10.5\,(1)} (see \cref{thm:DET-dual}).

The $\CP^2$-genus was previously studied by~\cites{yasuhara91,
  yasuhara92,aitnouh,pichelmeyer}. Specifically,
Yasuhara~\cites{yasuhara91,yasuhara92} studied the sliceness of
certain torus knots in $\CP^2$. Pichelmeyer~\cite{pichelmeyer} studied
the $\CP^2$-genus for low crossing knots and alternating knots. Ait
Nouh showed in \cite{aitnouh} that for $3 \leq q \leq 17$,
\[
g_{\CP^2}(T_{2, q}) = g_{4}(T_{2,q}) - 1 = \frac{q-3}2,
\]
and asked whether the equality
\[
g_{\CP^2}(T_{p,q}) = g_{4}(T_{p,q}) - 1 = \frac{(p-1)(q-1)}2-1
\]
holds for all $p, q >0$. \cref{thm:trick} shows that this is not the case, and yields the following corollary.

\begin{corollary}\label{corollary}
For every $n \geq 1$,
\[
g_{\CP^2}(T_{n,n-1}) \leq
\begin{cases}
g_{4}(T_{n,n-1}) - \frac{n-2}{2} & \textrm{if } n \equiv 0 \mod2 \\
g_{4}(T_{n,n-1}) - \frac{n-1}{2} & \textrm{if } n \equiv 1 \mod2.
\end{cases}
\]
\end{corollary}

Indeed the above inequality holds for the least genus of
null-homologous surfaces bounded by $T_{n,n-1}$, as we indicate in the
proof.

\subsection{Relative minimal genus problems in other 4-manifolds}
Given a closed, smooth $4$-manifold $M$, let $M^\times$ denote the
punctured manifold $M\sm \mathring{B^4}$. The \emph{$M$-genus} of a
knot $K\subseteq S^3$, denoted by $g_M(K)$, generalizes
the definition of Equation~\eqref{eq:mingenus} and is defined as
\[
g_M(K)=\min \{ g(\Sigma )\mid i\colon \Sigma \to M^{\times}, i(\partial \Sigma )=K\},
\]
where the minimum is taken over all
smooth, compact, orientable, properly embedded surface $\Sigma \subseteq
M^\times$ bounded by $K\subseteq \partial M^\times\cong S^3$. A knot
$K$ is said to be \emph{slice} in $M$ if $g_M(K)=0$. One may also
consider the \emph{topological $M$-genus}, denoted by $g^\top_M(K)$,
the least genus of a surface $F$ as above which is only locally flat
and embedded. Note that $g_M(K) \leq g_{S^4}(K)$ and $g^\top_M(K) \leq
g^\top_{S^4}(K)$ for all $M$.

The smooth and topological $S^4$-genera correspond to the usual smooth
and topological slice genera of knots and have been studied
extensively. In particular, there exist infinitely many knots with
trivial topological $S^4$-genus and nontrivial smooth
$S^4$-genus~\cites{Gom86,endo}. Any such knot can be used to produce a
nonstandard smooth structure on $\RR^4$~\cite{gompf-infinite} and in
general slicing knots in $S^4$ is connected to major open questions in
$4$-manifold topology, such as the smooth Poincar\'e conjecture and
the topological surgery
conjecture~\cites{man-and-machine,casson-freedman-atomic} (see
also~\cite{DET-book-enigmata}).  Slicing knots in more general
$4$-manifolds has also been fruitful, e.g.\ in revealing structure
within the knot concordance group~\cites{COT1, COT2, CT, CHL09,
  CHL11}, and in distinguishing between smooth concordance classes of
topologically slice
knots~\cites{CHH,cochran-horn,cha-kim-bipolar}. In~\cite{manolescu-marengon-piccirillo}
it was shown that the set of knots which bound smooth, null-homologous
disks in a $4$-manifold $M$ can distinguish between smooth structures
on $M$, that is,
there are examples of homeomorphic smooth
4-manifolds $M_1, M_2$ and a knot $K\subset S^3$ which bounds a
smooth, null-homologous disk in $M_1^\times$, but does not bound such a
disk in $M_2^\times$.
It is an open question whether the set of slice knots can
distinguish between smooth structures on a manifold.

It was shown in~\cite{KPRT}*{Corollary~1.12} using Freedman's disk
embedding theorem~\cites{F,FQ} that $g^\top_M(K)=0$ for every closed,
simply connected $4$-manifold $M$ other than $S^4$, $\CP^2$,
$\ol{\CP^2}$, $*\CP^2$, and $*\ol{\CP^2}$.
Here $\ol{M}$ denotes the 4-manifold $M$ with its orientation
reversed, and $*\CP ^2$ is the topological 4-manifold homotopy
equivalent, but not homeomorphic, to $\CP ^2$, constructed by Freedman in~\cite{F}. 
As mentioned before,
$g^\top_{\CP^2}(K)\leq 1$ for every knot $K$. Moreover,
$g^\top_{*\CP^2}(K)\leq 1$ for every knot $K$ as well. Further,
$g^\top_{\ol{\CP^2}}(K)=g^\top_{\CP^2}(-{K})$ for every knot $K$,
where $-{K}$ is the mirror image.  Therefore, the topological
problem for closed, simply connected $4$-manifolds with positive second
Betti number is
better understood. By contrast, many open questions remain in the smooth
setting. It is a straightforward consequence of the Norman trick that
all knots in $S^3$ are slice in both $S^2\times S^2$ and $S^2\wt\times
S^2 \cong \CP^2 \#
\overline{\CP^2}$~\citelist{\cite{norman-trick}*{Corollary~3 and
    Remark}\cite{suzuki}*{Theorem 1}} (see
also~\cites{ohyama,livingston-nullhom}).  By Theorem~\ref{thm:arb-large}
the
$\CP^2$-genus can be arbitrarily large.  It is open whether there is a
knot that is not slice in the $K3$-surface -- if such a knot exists,
its unknotting number must be more than $21$~\cite{marengon-mihajlovic}.

\subsection{Minimal genus problem in \texorpdfstring{$\CP^2\#\CP^2$}{CP2\#CP2}}
\label{sec:absolute2CP2}
Above we provided a possible extension of the Thom conjecture from
$\CP ^2$ to $(\CP ^2)^\times$, by replacing closed surfaces with
surfaces having boundary knotted in $S^3=\partial (\CP ^2)^\times$.
Another extension of Kronheimer-Mrowka's seminal result on $G_{\CP
  ^2}$ would be the determination of the corresponding function for
other closed 4-manifolds, for example $G_{\CP ^2\# \CP ^2}$ of $\CP
^2\# \CP ^2$.  (The cases of $\CP ^2 \# {\overline {\CP}}^2$ and
$S^2\times S^2$ were resolved by Ruberman in \cite{Ruberman}.)
Gauge-theoretic methods are less effective for this $4$-manifold --
the lower bounds, resting on variants of Furuta's
$\frac{10}{8}$-theorem, from \cite{Bryan} provide sharp results only
in a handful of cases (see \cite{Bryan}*{Corollary~1.7}), while
constructions are rare.

Let $(n,d)=nh_1+dh_2 \in H_2 (\CP ^2\# \CP ^2; \Z )\cong \Z \oplus \Z$
denote a second homology class in $\CP ^2\# \CP ^2$, with the
convention that $h_i$ is a generator in the $i^{th}$ summand and
$n,d\in \Z$.  A na{\"\i}ve approach to construct a surface
representing the class $(n,d)$ is to take the connected sum of the
genus-minimizing surfaces in the two $\CP^2$-components (representing
the desired homology class), giving a surface of genus $G_{\CP
  ^2}(n\cdot h)+G_{\CP ^2}(d\cdot h)$.  This straighforward upper
bound has been improved by 1 in \cite{aitnouh-2cp2} for an infinite
family of homology classes; below we implement roughly the same idea
to show that the difference between the value of $G_{\CP ^2\# \CP^2}$
can arbitrarily differ from the above na{\"\i}ve bound.  For all $k
\geq 0$, let $\varG(k) := (k-1)(k-2)/2$. This function agrees with
$G_{\CP^2}(k\cdot h)$ for all $k$ except $0$, where it is off by $1$.

\begin{theorem}\label{thm:2cp2}
  Suppose that $(n,d)\in H_2(\CP ^2\# \CP ^2; \Z )$ is a homology class
  with $n>d\geq 0$.
  \begin{itemize}
  \item For $n\equiv d \mod{2}$ we have
    \[
    \left(\varG(n)+\varG(d)\right)-G_{\CP^2\#\CP^2}(n,d)\geq \frac{n-3d}{2};
    \]
      \item for $n\not\equiv d \mod{2}$ we have
\[
    \left(\varG(n)+\varG(d)\right)-G_{\CP^2\#\CP^2}(n,d)\geq \frac{n-3d+1}{2}.
    \]
  \end{itemize}
\end{theorem}

Note that this bound is significant only when the right hand side is
positive; in this case it shows that for suitably chosen $(n,d)$ the
difference between $G_{\CP ^2\# \CP ^2}(n,d)$ and the na{\"i}ve guess
$G_{\CP ^2}(n)+G_{\CP ^2}(d)$ can be arbitrarily large. An
independent proof of Theorem \ref{thm:2cp2} will appear in a
forthcoming paper of Stefan Mihajlovi\'c.

\subsection*{Acknowledgements}

We are grateful to Stefan Mihajlovi\'c for pointing out an off-by-one error in the $d=0$ case of Theorem \ref{thm:2cp2} in a previous version of the paper, and to the anonymous referee for their comments.

MM acknowledges that: This project has received funding from the
European Union’s Horizon 2020 research and innovation programme under
the Marie Sk{\l}odowska-Curie grant agreement No.\ 893282.

 AS was partially supported by the \emph{\'Elvonal (Frontier) grant}
 KKP144148 of the NKFIH.  This material is partly based upon work
 supported by the National Science Foundation under Grant
 No.\ DMS-1928930, while AS was in residence at the Simons Laufer
 Mathematical Sciences Institute (previously known as MSRI) Berkeley,
 California, during the Fall 2022 semester.
 
\section{Genera of knots}
Given a closed 4-manifold $M$, we identify $H_2(M^\times, \partial (M^\times);\Z) \cong H_2(M^\times;\Z) \cong H_2(M, \Z)$.
We will use the following results from~\cites{viro70,gilmer81,Ozsvath-Szabo:2003-1}. 
\begin{lemma}\label{lemma}
Let $K\subseteq S^3$ be a knot which bounds a smooth, compact, connected, orientable, properly embedded surface $\Sigma$ of genus $g$ in $(\CP^2)^\times$ such that $[\Sigma]=d[\CP^1] $ in $H_2((\CP^2)^\times;\Z)$.  
\begin{enumerate}
\item \cite{Ozsvath-Szabo:2003-1}*{Theorem~1.1} For the $\tau$-invariant from Heegaard Floer homology, we have \[g \geq -\tau(K) + \frac{|d|(1-|d|)}{2}.\]
\item \cites{gilmer81,viro70} If $d$ is even, then \[2g+1 \geq \left| \frac{d^2}{2} - 1- \sigma(K)\right|.\] 
\item \cites{gilmer81,viro70} If $d$ is divisible by an odd prime $p$,  then 
  \[2g+1 \geq \left| \frac{p^2-1}{2p^2} d^2 - 1 - \sigma_{p}(K) \right|,\]  where $\sigma_p(K):= \sigma_K\left(e^{\pi i \frac{p-1}{p}}\right)$
  and $\sigma_K$ denotes the Levine-Tristram signature function.
\end{enumerate}
\end{lemma}

First we construct knots with arbitrarily large $\CP^2$-genus, which is the main ingredient in the proof of \cref{thm:top-v-smooth}.

\begin{proposition}\label{prop:arb-large}
Fix $g_0\geq 0$ and $c_0>\frac{3}{2} \sqrt{2g_0+2}>1$. Let $K$ be a knot with vanishing  Levine-Tristram signature function and such that the $\tau$-invariant from Heegaard Floer homology satisfies \[-\tau(K)\geq g_0- \frac{c_0(1-c_0)}{2}.\] Then $g_{\CP^2}(K)\geq g_0$. 
\end{proposition}

\begin{proof}
Let $\Sigma$ be a genus $g$ surface in  $(\CP^2)^{\times}$  with $\partial \Sigma=K$ and such that $[\Sigma]=d[\CP^1] $ in second homology. By~\cref{lemma}\,(1), we have that 
\begin{align*}
 g &\geq -\tau(K)+ \frac{|d|(1-|d|)}{2} \geq g_0 - \frac{c_0(1-c_0)}{2}+ \frac{|d|(1-|d|)}{2}.
\end{align*}
Then either $g\geq g_0$, in which case we are done, or $g < g_0$ and so $|d| > c_0>1$.  

So assume that $|d| > c_0$ and let $p$ be a prime factor of $|d|$. If $p=2$,  then by~\cref{lemma}\,(2), 
\[ 2g+1 \geq \frac{d^2}{2} - 1 > \frac{c_0^2}{2}-1 > \frac{9}{4} (2g_0+2)-1 > 2g_0+1. 
\]
It follows that $g \geq g_0$, as desired. 

If $p$ is odd, then $\frac{1}{2}- \frac{1}{2p^2} \geq\frac{1}{2}- \frac{1}{2(3)^2}= \frac{4}{9}$ and so by~\cref{lemma}\,(3) 
\[2g+1 \geq \frac{p^2-1}{2p^2} d^2 - 1 > \frac{p^2-1}{2p^2} c_0^2 -1 >\left(\frac{1}{2} -\frac{1}{2p^2}\right)\frac{9}{4}(2g_0+2)-1 \geq 2g_0+1. 
\]
It follows that $g \geq g_0$,  as desired. 
\end{proof}

\begin{proof}[Proof of~\cref{thm:top-v-smooth}]
It suffices to find topologically slice knots satisfying the hypothesis of \cref{prop:arb-large}. There are many such knots, e.g.\ one could take connected sums of a knot with trivial Alexander polynomial and negative $\tau$-invariant.
As a concrete example, let $K$ be the untwisted, negative clasped Whitehead double of the left-handed trefoil knot. By \cite{Hedden}*{Theorem 1.5}, we know that $\tau(K) = -1$, and since $\Delta_K(t) = 1$ it is topologically slice~\cite{FQ}*{Section~11.7}, hence its Levine-Tristram signatures all vanish. Then the knots $\#^\ell K$, for $\ell \gg 0$, satisfies the hypothesis of \cref{prop:arb-large} with $g_0=n$.
\end{proof}
\begin{remark}
  If one only desires a knot $K$ with $g_{\CP^2}^\top(K)=0$ and
  $g_{\CP^2}(K)\geq n$ for a given $n$, then
  by~\cref{thm:arf-zero,prop:arb-large} any knot $K$ with trivial Arf
  invariant, vanishing Levine-Tristram signature function, and
  satisfying $-\tau(K)\geq n- \frac{c_0(1-c_0)}{2}$ for
  $c_0>\frac{3}{2} \sqrt{2n+2}$, will work.
\end{remark}

\begin{remark}
  In~\cite{pichelmeyer}*{Conjecture~1}, Pichelmeyer conjectured that
  two knots with equal ordinary signature and Arf invariant must have
  equal $g_{\CP^2}$. The knots from \cref{thm:arb-large} with
  $g_{\CP^2}^\top(K_n) =0$ and $g_{\CP^2}(K_n)\geq n$ ($n\to \infty$)
  contradict the conjecture: since they are topologically slice in
  $B^4$, they have equal (indeed trivial) Arf invariant and
  Levine-Tristram signature function, and one can choose an infinite
  subfamily realizing distinct $\CP^2$-genera.
\end{remark}
Next we will prove \cref{thm:arf-zero} from the
introduction. For this we will need the following formulation of the
Arf invariant of a knot. Recall that the \emph{self-intersection
  number} of a properly, generically immersed disk $f\colon
(D^2,\partial D^2)\looparrowright (M,\partial M)$ restricting to an
embedding on the boundary, where $M$ is a simply connected
$4$-manifold, is the signed count of double points of $f$.
Here and in the rest of the paper, the symbol $\looparrowright$ denotes a \emph{generic immersion} of a compact surface, i.e.\ an immersion such that the singular set is a closed, discrete subset of $M$ consisting only of transverse double points, each of whose preimages lies in the interior of the surface. We will only use this notion where the target manifold is smooth, but we note that there is an analogous notion in the topological setting~\citelist{\cite{FQ}*{Chapter~1}\cite{Freedman-book-introDET}*{Section~11.1}}.  We will also use (possibly immersed) Whitney disks in the argument. For more details about Whitney disks, see, for example, \cite{Freedman-book-introDET}.

\begin{proposition}[\citelist{\cite{matsumoto78}\cite{freedman-kirby}\cite{CST-twisted-whitney}*{Lemma~10}}]\label{prop:arf-whitney}
Let $K\subseteq S^3$ be a knot bounding a generically immersed disk $\Delta$ in $B^4$ with trivial self-intersection number. This implies that there is a collection $\{W_i\}$ of framed, generically immersed Whitney disks pairing up the self-intersections of $\Delta$ and intersecting $\Delta$ in transverse double points in the interiors $\{\mathring{W}_i\}$.

For any such collection $\{W_i\}$, we have the equality 
\[\Arf(K) = \Sum_i \Delta \cdot \mathring{W}_i \mod{2}.\]
\end{proposition}

We will also need the following result of Freedman-Quinn, giving a
sufficient condition under which an immersion of a disk is homotopic to an embedding.\footnote{For the convenience of the reader
  we have stated~\cite{FQ}*{Theorem~10.5\,(1)} in the special case
  where the ambient manifold is simply connected. The more general
  formulation in~\cite{FQ} had an error, which was detected and
  corrected by Stong in~\cite{Stong}. The error is related to elements
  of order two in the fundamental group of the ambient manifold and is
  not relevant here. For more details, see~\cites{Stong,KPRT}.}  Since
the statement in \cite{FQ} does not match the one below exactly, we
point out that Theorem \ref{thm:DET-dual} is also an immediate
corollary of \cite{KPRT}*{Theorem 1.2}, using \cite{KPRT}*{Lemma 5.5}.

\begin{theorem}[\cite{FQ}*{Theorem~10.5\,(1)}]\label{thm:DET-dual}
Let $M$ be a simply connected $4$-manifold. Let $f\colon (D^2,\partial
D^2)\looparrowright (M,\partial M)$ be a proper, generic immersion
with trivial self-intersection number, restricting to an embedding on
the boundary, and admitting an algebraically dual sphere, i.e.\ there
is a sphere $C\colon S^2\looparrowright M$ with $f\cdot C\equiv
1\mod{2}$. Assume further that there exists a collection $\{W_i\}$ of
framed, generically immersed Whitney disks pairing up the
self-intersections of $f$ and intersecting the image of $f$ in
transverse double points in the interiors $\{\mathring{W}_i\}$ so that
$\sum_i f\cdot \mathring{W}_i\equiv 0\mod{2}$.

Then $f$ is homotopic to a locally flat embedding relative to the boundary. 
\end{theorem}

Note that the above result differs from that proven by Freedman in~\cite{F} in that the algebraically dual 
sphere $C$ is not required to have trivial normal bundle, necessitating the additional condition on the 
Whitney disks.

\begin{proof}[Proof of \cref{thm:arf-zero}]
Let $\Delta$ be a generically immersed disk bounded by $K$ in $B^4$,
obtained e.g.\ as the trace of a null-homotopy. By adding local
self-intersections if necessary, we can assume that the signed count
of double points of $\Delta$ equals zero, i.e.~that $\Delta$ has
vanishing self-intersection number. Then these double points can be
paired by a collection of framed, generically immersed Whitney disks
$\{W_i\}$, whose interiors intersect $\Delta$ in isolated points (see, for example, \cite{Freedman-book-introDET}*{Proposition~11.10}). By
\cref{prop:arf-whitney}, we know that $0=\Arf(K) = \sum_i \Delta \cdot
\mathring{W}_i \mod{2}$.

Perform a connected sum of $B^4$ with $\CP^2$, to obtain $(\CP^2)^\times$. By choosing the location of the connected sum carefully, we ensure that $\Delta$ and $\{W_i\}$ also lie in $(\CP^2)^\times$, and are in particular disjoint from $\CP^1\subseteq (\CP^2)^\times$. Tube $\Delta$ into $\CP^1$, and call the result $\Delta'$. Since $\CP^1$ is embedded, the collection $\{W_i\}$ pairs up all the double points of $\Delta'$. 
A push-off $C$ of $\CP^1$ provides an algebraically dual sphere for $\Delta'$. Then by \cref{thm:DET-dual}, the disk $\Delta'$ is homotopic to a locally flat embedding, which is the desired slice disk. 
\end{proof}

\begin{remark}
The disk that we have constructed in the proof of~\cref{thm:arf-zero} is a generator of $H_2((\CP^2)^\times;\Z)$. By~\cite{KPRT}*{Theorem~1.6}, one sees that a knot $K$ bounds a locally flat, embedded disk in $(\CP^2)^\times$ representing a generator of $H_2((\CP^2)^\times;\Z)$ if and only if $\Arf(K)=0$ (see also~\cite{KPRT}*{Proof of Corollary~1.12}).  
\end{remark}

We finish this section by considering the $\CP^2$-genera of certain torus knots.

\begin{theorem}
\label{thm:trick}
Let $n > d \geq 0$.
\begin{itemize}
    \item If $n \equiv d \mod2$, the torus knot $T_{n,n-1}$ bounds a surface $\Sigma \subseteq (\CP^2)^\times$ with degree $d$ and genus 
\[
g(\Sigma) = \left(\frac{n+d-2}2\right)^2 + \left(\frac{n-d-2}2\right)^2.
\]
\item If $n \not\equiv d \mod2$, the torus knot $T_{n,n-1}$ bounds a surface $\Sigma \subseteq (\CP^2)^\times$ with degree $d$ and genus 
\[
g(\Sigma) = \left( \frac{n+d-1}{2} \right) \cdot \left( \frac{n+d-3}{2} \right) + \left( \frac{n-d-1}{2} \right) \cdot \left( \frac{n-d-3}{2} \right)
\]
\end{itemize}
\end{theorem}

The proof of Theorem \ref{thm:trick} is based on a certain torus knot
twisting technique, which has already appeared in the literature (see
e.g.\ \cite{McCoy}*{Figure 7} for a recent appearance).

\begin{proof}
For the case $n \equiv d \mod 2$, set $a = \frac{n+d}2$ and $b = \frac{n-d}2$. \cref{fig:trick1} shows a concordance from $K_0 := T_{a, 2a-1} \# T_{b, 2b-1}$ to $K_1 := T_{n,n-1}$ in $\CP^2\setminus (B^4 \sqcup B^4)$, with degree $d = a-b$.

\begin{figure}[h!]
    \centering
    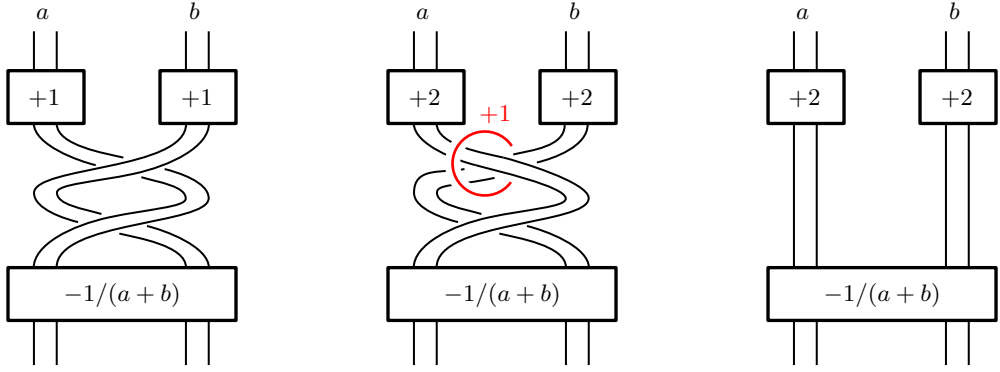
    \caption{The figure on the left shows the torus knot $T_{a+b,
        a+b-1}$.  A box with an integer denotes that many full twists
      (right-handed for positive, left-handed for negative integers).
      The box containing $-\tfrac{1}{a+b}$ indicates a negative
      fractional twist, i.e.\ the rightmost strand passes across to the
      extreme left, above all the other strands.  The second figure
      shows how to change $a \cdot b$ crossings with a blow-up, at the
      expense of adding a full twist on each of the two bundles of
      parallel strands.  The third figure shows a simplified diagram
      of the knot in the second figure, after one forgets the
      $(+1)$-framed $2$-handle. From the last figure it is clear that the knot in question is $T_{a, 2a-1}
      \# T_{b, 2b-1}$.}
    \label{fig:trick1}
\end{figure}

We can cap off this concordance with a minimum genus surface for $K_0$ in $B^4$, thus obtaining the desired surface $\Sigma$, with genus
\[
g(\Sigma) = g_4(T_{a, 2a-1}) + g_4(T_{b, 2b-1}) = (a-1)^2 + (b-1)^2.
\]

For the case $n \not\equiv d \mod 2$, set $a = \frac{n+d-1}2$ and $b = \frac{n-d-1}2$. Instead of the concordance in \cref{fig:trick1}, we now use the one illustrated in \cref{fig:trick2}. 
\begin{figure}
    \centering
    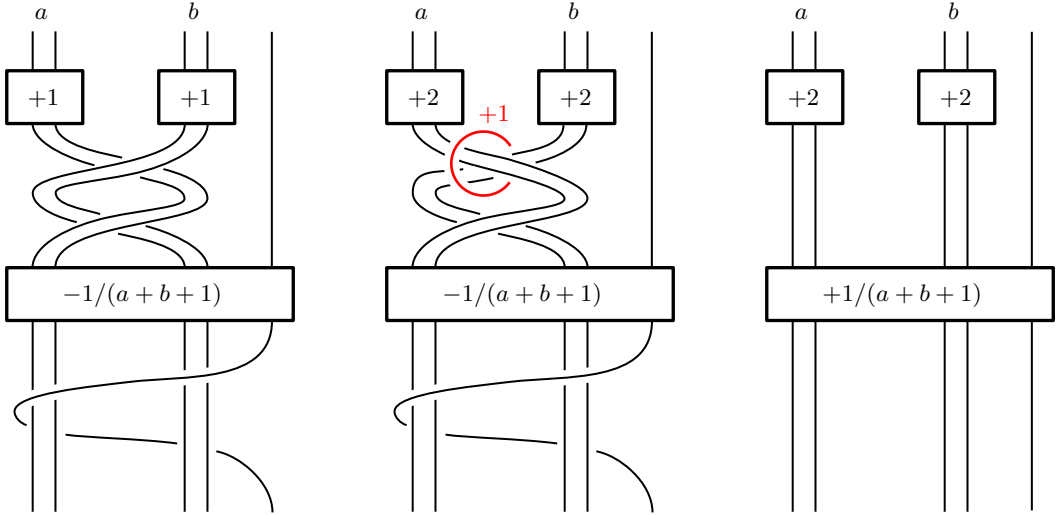
    \caption{The figure on the left shows the torus knot $T_{a+b+1, a+b}$.
    The second figure shows how to change $a \cdot b$ crossing with a blow-up, at the expense of adding a full twist on each of the two bundles of parallel strands.
    After simplifying the diagram (third figure), the knot is identified with $T_{a, 2a+1} \# T_{b, 2b+1} \# T_{1,0} = T_{a, 2a+1} \# T_{b, 2b+1}$. (Note the sign change in the fractional twist in the third figure.)}
    \label{fig:trick2}
\end{figure}
As before, we cap off with a minimum genus surface for $K_0$ in $B^4$, thus obtaining the desired surface $\Sigma$, which now has genus
\[
g(\Sigma) = g_4(T_{a, 2a+1}) + g_4(T_{b, 2b+1}) = a\cdot (a-1) + b\cdot (b-1).\qedhere
\]
\end{proof}

\begin{proof}[Proof of \cref{corollary}]
The corollary follows by applying Theorem \ref{thm:trick} with $d = 0$.

More explicitly, if $n \equiv 0 \mod2$, Theorem \ref{thm:trick} provides a smooth surface $\Sigma$ in $(\CP^2)^\times$ of genus
\[g(\Sigma) = \frac{(n-2)(n-2)}2,\]
which is an upper bound to $g_{\CP^2}(T_{n,n-1})$.
Thus, the difference $g_4(T_{n,n-1}) - g_{\CP^2}(T_{n,n-1})$ is at least
\[
g_4(T_{n,n-1}) - g(\Sigma) = \frac{(n-1)(n-2)}2 - \frac{(n-2)(n-2)}2 = \frac{n-2}2 \cdot((n-1)-(n-2)) = \frac{n-2}2.
\]

The case $n \equiv 1 \mod2$ is very similar: in this case Theorem \ref{thm:trick} provides a smooth surface $\Sigma$ in $(\CP^2)^\times$ of genus
\[g(\Sigma) = \frac{(n-1)(n-3)}2,\]
and therefore the difference $g_4(T_{n,n-1}) - g_{\CP^2}(T_{n,n-1})$ is at least
\[
g_4(T_{n,n-1}) - g(\Sigma) = \frac{(n-1)(n-2)}2 - \frac{(n-1)(n-3)}2 = \frac{n-1}2 \cdot((n-2)-(n-3)) = \frac{n-1}2.
\qedhere
\]
\end{proof}

\section{The genus function in \texorpdfstring{$\CP^2 \# \CP^2$}{CP2\#CP2}}
Using~\cref{thm:trick}, it is quite straighforward to complete the proof of~\cref{thm:2cp2}.
The strategy is the same as for the proof of \cite{aitnouh-2cp2}*{Theorem~1.2}.
    
\begin{proof}[Proof of Theorem~\ref{thm:2cp2}]
We first note that the knot $-T_{n,n-1}$, the mirror of the 
$(n,n-1)$ torus knot, is slice in $\CP ^2$, i.e., $g_{\CP
  ^2}(-T_{n,n-1})=0$, and that in addition the slice disk can be chosen to
represent $n$-times the generator in relative homology.
To see this, start from a description of the unknot as the closure of the braid
$
\sigma_1 \sigma_2 \cdots \sigma_{n-1},
$
as in \cref{fig:cable}, which (being unknotted) bounds a disk in $S^3$, and hence in $B^4$.
If we attach a $(+1)$-framed 2-handle to $B^4$ along the red unknot $U$, which encircles all the strands of the braid once, then we obtain a Kirby diagrammatic description of $(\CP^2)^\times$, together with a knot $K$ sitting in the new boundary $S^3 = \partial\left( (\CP^2)^\times\right)$; note that by construction $K$ bounds a disk in $(\CP^2)^\times$, in homology class given by $\mathrm{lk}(K,U) = n$.
Upon doing a Rolfsen twist, $K$ is revealed to be the knot $-T_{n,n-1}$.

\begin{figure}[htb]
    \centering
\begin{tikzpicture}
        \node[anchor=south west,inner sep=0] at (0,0){	\includegraphics{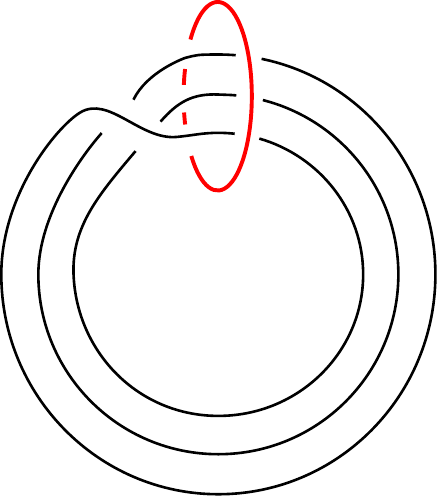}};
		\node at (2.7,5) {\color{red}$+1$};
		\node at (1.8,5) {\color{red}$U$};
		\node at (4.7,2) {$K$};
	\end{tikzpicture}
\caption{Attaching a $(+1)$-framed 2-handle to $B^4$ along an unknot gives the standard handle decomposition for $(\CP^2)^\times$.  In the resulting surgery description of $S^3$,   the knot $K$ depicted above appears unknotted,  though blowing down the $(+1)$-surgery curve to get the empty surgery diagram for $S^3$ reveals that $K=-T_{n,n-1}$.  The image of the standard slice disk for the unknot in $B^4$ in $(\CP^2)^\times$ is the desired slice disk for $K$,  and represents $n$-times the generator of relative second homology. The case $n=3$ is shown.}
\label{fig:cable}
\end{figure}
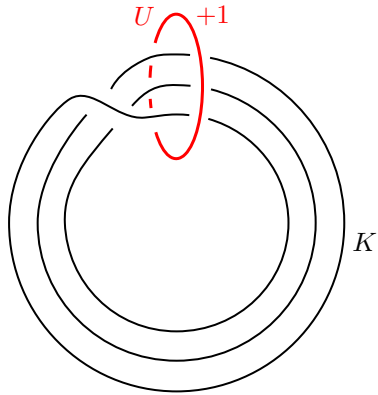    

Taking the boundary connected sum of $((\CP ^2)^{\times},T_{n,n-1})$
and $((\CP ^2)^{\times},-T_{n,n-1})$ together with the surface
described in the proof of Theorem~\ref{thm:trick} in the first and the
previously described slice disk in the second summand (and taking the
connected sum at points of the knots) we get a surface in $(\CP ^2\#
\CP ^2)^{\times}$ with genus given in Theorem~\ref{thm:trick}. In
addition, the boundary of this surface is the connected sum
$T_{n,n-1}\#-T_{n,n-1}$, which is a slice knot in $B^4$ (as is any
knot of the form $K\# -K$). This shows that the above surface with
boundary can be capped off by a disk to provide a closed surface in
$\CP^2 \# \CP^2$ with the same genus, representing the homology class
$(n,d)$. The proof of the theorem now follows after some simple arithmetic.
\end{proof}

\bibliographystyle{alpha}
\bibliography{bib}

\begin{bibdiv}
\begin{biblist}

\bib{aitnouh}{article}{
      author={Ait~Nouh, Mohamed},
       title={Genera and degrees of torus knots in {$\mathbb{CP}^2$}},
        date={2009},
        ISSN={0218-2165},
     journal={J. Knot Theory Ramifications},
      volume={18},
      number={9},
       pages={1299\ndash 1312},
         url={https://doi.org/10.1142/S0218216509007439},
      review={\MR{2569563}},
}

\bib{aitnouh-2cp2}{article}{
      author={Ait~Nouh, Mohamed},
       title={The minimal genus problem in {$\mathbb{CP}^2\#\mathbb{CP}^2$}},
        date={2014},
        ISSN={1472-2747},
     journal={Algebr. Geom. Topol.},
      volume={14},
      number={2},
       pages={671\ndash 686},
         url={https://doi.org/10.2140/agt.2014.14.671},
      review={\MR{3159966}},
}

\bib{Bryan}{article}{
      author={Bryan, Jim},
       title={Seiberg-{W}itten \`a la {F}uruta and genus bounds for classes
  with divisibility},
        date={1997},
        ISSN={1300-0098},
     journal={Turkish J. Math.},
      volume={21},
      number={1},
       pages={55\ndash 59},
      review={\MR{1456159}},
}

\bib{casson-freedman-atomic}{incollection}{
      author={Casson, Andrew},
      author={Freedman, Michael},
       title={Atomic surgery problems},
        date={1984},
   booktitle={Four-manifold theory ({D}urham, {N}.{H}., 1982)},
      series={Contemp. Math.},
      volume={35},
   publisher={Amer. Math. Soc., Providence, RI},
       pages={181\ndash 199},
         url={https://doi.org/10.1090/conm/035/780579},
      review={\MR{780579}},
}

\bib{cochran-horn}{article}{
      author={Cochran, Tim~D.},
      author={Horn, Peter~D.},
       title={Structure in the bipolar filtration of topologically slice
  knots},
        date={2015},
        ISSN={1472-2747},
     journal={Algebr. Geom. Topol.},
      volume={15},
      number={1},
       pages={415\ndash 428},
         url={https://doi.org/10.2140/agt.2015.15.415},
      review={\MR{3325742}},
}

\bib{CHH}{article}{
      author={Cochran, Tim~D.},
      author={Harvey, Shelly},
      author={Horn, Peter},
       title={Filtering smooth concordance classes of topologically slice
  knots},
        date={2013},
        ISSN={1465-3060},
     journal={Geom. Topol.},
      volume={17},
      number={4},
       pages={2103\ndash 2162},
         url={https://doi.org/10.2140/gt.2013.17.2103},
      review={\MR{3109864}},
}

\bib{CHL09}{article}{
      author={Cochran, Tim~D.},
      author={Harvey, Shelly},
      author={Leidy, Constance},
       title={Knot concordance and higher-order {B}lanchfield duality},
        date={2009},
        ISSN={1465-3060},
     journal={Geom. Topol.},
      volume={13},
      number={3},
       pages={1419\ndash 1482},
         url={https://doi.org/10.2140/gt.2009.13.1419},
      review={\MR{2496049}},
}

\bib{CHL11}{article}{
      author={Cochran, Tim~D.},
      author={Harvey, Shelly},
      author={Leidy, Constance},
       title={Primary decomposition and the fractal nature of knot
  concordance},
        date={2011},
        ISSN={0025-5831},
     journal={Math. Ann.},
      volume={351},
      number={2},
       pages={443\ndash 508},
         url={https://doi.org/10.1007/s00208-010-0604-5},
      review={\MR{2836668}},
}

\bib{cha-kim-bipolar}{article}{
      author={Cha, Jae~Choon},
      author={Kim, Min~Hoon},
       title={The bipolar filtration of topologically slice knots},
        date={2021},
        ISSN={0001-8708},
     journal={Adv. Math.},
      volume={388},
       pages={Paper No. 107868, 32},
         url={https://doi.org/10.1016/j.aim.2021.107868},
      review={\MR{4283760}},
}

\bib{COT1}{article}{
      author={Cochran, Tim~D.},
      author={Orr, Kent~E.},
      author={Teichner, Peter},
       title={Knot concordance, {W}hitney towers and {$L^2$}-signatures},
        date={2003},
        ISSN={0003-486X},
     journal={Ann. of Math. (2)},
      volume={157},
      number={2},
       pages={433\ndash 519},
         url={https://doi.org/10.4007/annals.2003.157.433},
      review={\MR{1973052}},
}

\bib{COT2}{article}{
      author={Cochran, Tim~D.},
      author={Orr, Kent~E.},
      author={Teichner, Peter},
       title={Structure in the classical knot concordance group},
        date={2004},
        ISSN={0010-2571},
     journal={Comment. Math. Helv.},
      volume={79},
      number={1},
       pages={105\ndash 123},
         url={https://doi.org/10.1007/s00014-001-0793-6},
      review={\MR{2031301}},
}

\bib{CST-twisted-whitney}{article}{
      author={Conant, J.},
      author={Schneiderman, R.},
      author={Teichner, P.},
       title={Milnor invariants and twisted {W}hitney towers},
        date={2014},
     journal={J. Topol.},
      volume={7},
      number={1},
       pages={187\ndash 224},
}

\bib{CT}{article}{
      author={Cochran, Tim~D.},
      author={Teichner, Peter},
       title={Knot concordance and von {N}eumann {$\rho$}-invariants},
        date={2007},
        ISSN={0012-7094},
     journal={Duke Math. J.},
      volume={137},
      number={2},
       pages={337\ndash 379},
         url={https://doi.org/10.1215/S0012-7094-07-13723-2},
      review={\MR{2309149}},
}

\bib{endo}{article}{
      author={Endo, Hisaaki},
       title={Linear independence of topologically slice knots in the smooth
  cobordism group},
        date={1995},
        ISSN={0166-8641},
     journal={Topology Appl.},
      volume={63},
      number={3},
       pages={257\ndash 262},
         url={https://doi.org/10.1016/0166-8641(94)00062-8},
      review={\MR{1334309}},
}

\bib{man-and-machine}{article}{
      author={Freedman, Michael},
      author={Gompf, Robert},
      author={Morrison, Scott},
      author={Walker, Kevin},
       title={Man and machine thinking about the smooth 4-dimensional
  {P}oincar\'{e} conjecture},
        date={2010},
        ISSN={1663-487X},
     journal={Quantum Topol.},
      volume={1},
      number={2},
       pages={171\ndash 208},
         url={https://doi.org/10.4171/QT/5},
      review={\MR{2657647}},
}

\bib{freedman-kirby}{inproceedings}{
      author={Freedman, Michael},
      author={Kirby, Robion},
       title={A geometric proof of {R}ochlin's theorem},
        date={1978},
   booktitle={Algebraic and geometric topology ({P}roc. {S}ympos. {P}ure
  {M}ath., {S}tanford {U}niv., {S}tanford, {C}alif., 1976), {P}art 2},
      series={Proc. Sympos. Pure Math., XXXII},
   publisher={Amer. Math. Soc., Providence, R.I.},
       pages={85\ndash 97},
}

\bib{FQ}{book}{
      author={Freedman, Michael},
      author={Quinn, Frank},
       title={Topology of $4$-manifolds},
      series={Princeton Mathematical Series},
   publisher={Princeton University Press},
        date={1990},
      volume={39},
}

\bib{F}{article}{
      author={Freedman, Michael},
       title={The topology of four-dimensional manifolds},
        date={1982},
     journal={J. Differential Geom.},
      volume={17},
      number={3},
       pages={357\ndash 453},
}

\bib{gilmer81}{article}{
      author={Gilmer, Patrick~M.},
       title={Configurations of surfaces in {$4$}-manifolds},
        date={1981},
        ISSN={0002-9947},
     journal={Trans. Amer. Math. Soc.},
      volume={264},
      number={2},
       pages={353\ndash 380},
         url={https://doi.org/10.2307/1998544},
      review={\MR{603768}},
}

\bib{gompf-infinite}{article}{
      author={Gompf, Robert~E.},
       title={An infinite set of exotic {$\mathbb{R}^4$}'s},
        date={1985},
        ISSN={0022-040X},
     journal={J. Differential Geom.},
      volume={21},
      number={2},
       pages={283\ndash 300},
         url={http://projecteuclid.org/euclid.jdg/1214439566},
      review={\MR{816673}},
}

\bib{Gom86}{article}{
      author={Gompf, Robert~E.},
       title={Smooth concordance of topologically slice knots},
        date={1986},
        ISSN={0040-9383},
     journal={Topology},
      volume={25},
      number={3},
       pages={353\ndash 373},
         url={https://doi.org/10.1016/0040-9383(86)90049-2},
      review={\MR{842430}},
}

\bib{Hedden}{article}{
      author={Hedden, Matthew},
       title={Knot {F}loer homology of {W}hitehead doubles},
        date={2007},
        ISSN={1465-3060,1364-0380},
     journal={Geom. Topol.},
      volume={11},
       pages={2277\ndash 2338},
         url={https://doi.org/10.2140/gt.2007.11.2277},
      review={\MR{2372849}},
}

\bib{thomconjecture}{article}{
      author={Kronheimer, P.~B.},
      author={Mrowka, T.~S.},
       title={The genus of embedded surfaces in the projective plane},
        date={1994},
        ISSN={1073-2780},
     journal={Math. Res. Lett.},
      volume={1},
      number={6},
       pages={797\ndash 808},
         url={https://doi.org/10.4310/MRL.1994.v1.n6.a14},
      review={\MR{1306022}},
}

\bib{DET-book-enigmata}{incollection}{
      author={Kim, Min~Hoon},
      author={Orson, Patrick},
      author={Park, JungHwan},
      author={Ray, Arunima},
       title={Open problems},
        date={2021},
   booktitle={The disc embedding theorem},
   publisher={Oxford Univ. Press, Oxford},
       pages={353\ndash 382},
      review={\MR{4519521}},
}

\bib{KPRT}{unpublished}{
      author={Kasprowski, Daniel},
      author={Powell, Mark},
      author={Ray, Arunima},
      author={Teichner, Peter},
       title={Embedding surfaces in 4-manifolds},
        date={2022},
        note={ar{X}iv: 2201.03961},
}

\bib{livingston-nullhom}{incollection}{
      author={Livingston, Charles},
       title={Null-homologous unknottings},
        date={2021},
   booktitle={Topology and geometry---a collection of essays dedicated to
  {V}ladimir {G}. {T}uraev},
      series={IRMA Lect. Math. Theor. Phys.},
      volume={33},
   publisher={Eur. Math. Soc., Z\"{u}rich},
       pages={59\ndash 68},
         url={https://doi.org/10.4171/IRMA/33-1/3},
      review={\MR{4394501}},
}

\bib{matsumoto78}{inproceedings}{
      author={Matsumoto, Yukio},
       title={Secondary intersectional properties of {$4$}-manifolds and
  {W}hitney's trick},
        date={1978},
   booktitle={Algebraic and geometric topology ({P}roc. {S}ympos. {P}ure
  {M}ath., {S}tanford {U}niv., {S}tanford, {C}alif., 1976), {P}art 2},
      series={Proc. Sympos. Pure Math., XXXII},
   publisher={Amer. Math. Soc., Providence, R.I.},
       pages={99\ndash 107},
}

\bib{McCoy}{article}{
      author={McCoy, Duncan},
       title={Null-homologous twisting and the algebraic genus},
        date={2021},
     journal={2019-20 MATRIX Annals},
       pages={147\ndash 165},
}

\bib{marengon-mihajlovic}{unpublished}{
      author={Marengon, Marco},
      author={Mihajlovi{\'c}, Stefan},
       title={Unknotting number 21 knots are slice in {K}3},
        date={2022},
        note={Preprint, available at
  \href{https://arxiv.org/abs/2210.10089}{ar{X}iv:2210.10089}},
}

\bib{manolescu-marengon-piccirillo}{unpublished}{
      author={Manolescu, Ciprian},
      author={Marengon, Marco},
      author={Piccirillo, Lisa},
       title={Relative genus bounds in indefinite four-manifolds},
        date={2020},
        note={ar{X}iv:2012.12270},
}

\bib{norman-trick}{article}{
      author={Norman, R.~A.},
       title={Dehn's lemma for certain {$4$}-manifolds},
        date={1969},
        ISSN={0020-9910},
     journal={Invent. Math.},
      volume={7},
       pages={143\ndash 147},
         url={https://doi.org/10.1007/BF01389797},
      review={\MR{246309}},
}

\bib{ohyama}{article}{
      author={Ohyama, Yoshiyuki},
       title={Twisting and unknotting operations},
        date={1994},
        ISSN={0214-3577},
     journal={Rev. Mat. Univ. Complut. Madrid},
      volume={7},
      number={2},
       pages={289\ndash 305},
      review={\MR{1297516}},
}

\bib{Ozsvath-Szabo:2003-1}{article}{
      author={Ozsv{\'a}th, P.},
      author={Szab{\'o}, Z.},
       title={Knot {F}loer homology and the four-ball genus},
        date={2003},
        ISSN={1465-3060},
     journal={Geom. Topol.},
      volume={7},
       pages={615\ndash 639},
      review={\MR{MR2026543 (2004i:57036)}},
}

\bib{pichelmeyer}{unpublished}{
      author={Pichelmeyer, Jacob},
       title={Genera of knots in the complex projective plane},
        date={2019},
        note={ar{X}iv:1912.01787},
}

\bib{Freedman-book-introDET}{incollection}{
      author={Powell, Mark},
      author={Ray, Arunima},
       title={Intersection numbers and the statement of the disc embedding
  theorem},
        date={2021},
   booktitle={The disc embedding theorem},
      editor={Behrens, Stefan},
      editor={Kalm\'{a}r, Boldizs\'{a}r},
      editor={Kim, Min~Hoon},
      editor={Powell, Mark},
      editor={Ray, Arunima},
   publisher={Oxford University Press},
}

\bib{Ruberman}{article}{
      author={Ruberman, Daniel},
       title={The minimal genus of an embedded surface of non-negative square
  in a rational surface},
        date={1996},
        ISSN={1300-0098},
     journal={Turkish J. Math.},
      volume={20},
      number={1},
       pages={129\ndash 133},
      review={\MR{1392668}},
}

\bib{Stong}{article}{
      author={Stong, Richard},
       title={Existence of {$\pi_1$}-negligible embeddings in {$4$}-manifolds.
  {A} correction to {T}heorem 10.5 of {F}reedman and {Q}uinn},
        date={1994},
        ISSN={0002-9939},
     journal={Proc. Amer. Math. Soc.},
      volume={120},
      number={4},
       pages={1309\ndash 1314},
      review={\MR{1215031}},
}

\bib{suzuki}{article}{
      author={Suzuki, Shin'ichi},
       title={Local knots of 2-spheres in 4-manifolds},
        date={1969},
     journal={Proceedings of the Japan Academy},
      volume={45},
      number={1},
       pages={34\ndash 38},
}

\bib{viro70}{article}{
      author={Viro, Oleg~Ya},
       title={Placements in codimension 2 and boundary},
        date={1975},
     journal={Uspekhi Mat. Nauk},
      volume={30},
       pages={231\ndash 232},
}

\bib{yasuhara91}{article}{
      author={Yasuhara, Akira},
       title={{$(2,15)$}-torus knot is not slice in {$\mathbb{CP}^2$}},
        date={1991},
        ISSN={0386-2194},
     journal={Proc. Japan Acad. Ser. A Math. Sci.},
      volume={67},
      number={10},
       pages={353\ndash 355},
         url={http://projecteuclid.org/euclid.pja/1195511928},
      review={\MR{1151354}},
}

\bib{yasuhara92}{article}{
      author={Yasuhara, Akira},
       title={On slice knots in the complex projective plane},
        date={1992},
        ISSN={0214-3577},
     journal={Rev. Mat. Univ. Complut. Madrid},
      volume={5},
      number={2-3},
       pages={255\ndash 276},
      review={\MR{1195083}},
}

\end{biblist}
\end{bibdiv}
\end{document}